\documentclass[11pt,reqno]{amsart}
\usepackage[cp1251]{inputenc}
\usepackage[english,russian]{babel}
\usepackage{amsmath}
\usepackage{amssymb}
\usepackage{amsfonts}
\usepackage{graphicx}
\usepackage{xcolor}
\usepackage[colorlinks]{hyperref}

\usepackage{titlesec}
\titlespacing{\section}{0pt}{2.5ex}{1.5ex}
\titlespacing{\subsection}{0pt}{1.5ex}{1ex}
\titlespacing{\subsubsection}{0pt}{1.5ex}{1ex}
\titleformat{\section}{\large\bfseries\centering}{\thesection}{1em}{}
\titleformat{\subsection}[runin]{\bfseries}{\thesubsection.}{0.5em}{}[.\mbox{\ }]
\titleformat{\subsubsection}[runin]{\bfseries}{\thesubsubsection.}{0.4em}{}[.\mbox{\ }]

\newtheorem{proposition}{Proposition}
\newtheorem{question}{Question}

\begin{document}
\renewcommand{\refname}{References}
\renewcommand{\proofname}{Proof.}
\renewcommand{\figurename}{Fig.}
\renewcommand{\tablename}{Table}

\thispagestyle{empty}

\begin{center}
{\bf \LARGE  On cubic graphs having the maximal \\[1mm] coalition number}

\bigskip

Andrey A. Dobrynin$^{1}$  and  Hamidreza Golmohammadi$^{1,2}$  

\bigskip

$^{1}$ {\it Sobolev Institute of Mathematics,  Siberian Branch of the Russian \\ Academy of Sciences,
pr. Koptyuga, 4, 630090, Novosibirsk, Russia}

$^{2}$ {\it Novosibirsk State University,  Pirogova str., 2, 630090, Novosibirsk, Russia}

{\rm dobr@math.nsc.ru, h.golmohammadi@g.nsu.ru}
\end{center}

\bigskip

\noindent{\bf Abstract.}
{\small
A coalition in a graph $G$ with a vertex set $V$  consists of two disjoint sets
$V_1, V_2\subset V$, such that neither $V_1$ nor $V_2$ is a dominating set, but the union
$V_1\cup V_2$ is a dominating set in $G$.
A partition of graph vertices is called a coalition partition $\mathcal{P}$
if every non-dominating set of $\mathcal{P}$ is a member of a coalition,
and every dominating set is a single-vertex set.
The coalition number $C(G)$ of a graph $G$ is the maximum cardinality of its coalition partitions. 
It is known that for cubic graphs $C(G) \le 9$.  The existence of cubic graphs with
the maximum coalition number is an unsolved problem.
In this paper, an infinite family of cubic graphs satisfying $C(G)=9$ is constructed.\medskip

}

\bigskip

\section{Introduction}
Throughout this paper, $G = (V, E)$  denotes an undirected, simple, and connected graph.
The vertex set of $G$ is denoted by $V(G)$, and the cardinality of $V(G)$
is called the order of $G$. The maximum degree of vertices of a graph $G$ is denoted by $\Delta(G)$.
A set $S \subseteq V$ is a dominating set if every vertex of $V-S$ is adjacent to at least one vertex in $S$.
Domination in graphs is a well-studied topic in graph theory, and the bibliography on this subject has been surveyed
in \cite{13,14}. There are different kinds of dominating sets, which have been explored, such as total,
connected, independent, double dominating sets, and so on \cite{5a,6,7,8a,15}.

A coalition is generally defined as a temporary alliance of two or more (political) parties to
work together to achieve a common goal.
In 2020, inspired by the idea that the union of two sets should have a property that neither
set has, the notion of coalitions in graphs have been introduced \cite{9},  and have subsequently 
been studied, for example,
in \cite{1,5,10,11,12}. A coalition in a graph $G$ consists of two disjoint sets of vertices
 $V_1$ and $V_2$, such that neither
$V_1$ nor $V_2$ is a dominating set but the union $V_1\cup V_2$ is a dominating set of $G$.
A coalition partition in a graph $G$ of order $n$ is a vertex partition
$\mathcal{P}=\{V_1, V_2,\dots, V_k\}$, such that every set $V_i$ either is a dominating
set consisting of a single-vertex of degree $n-1$, or is not a dominating set but
forms a coalition with another set $V_j$, which is not a dominating set.
The  coalition number $C(G)$ of a graph $G$ equals the maximum order $k$ of a coalition partition.
Coalitions in graphs, based on various types of dominating sets, have been recently
studied \cite{2,3,4,8}.

Haynes et al. \cite{11} proved that, for any graph, $C(G)\leq\ (\Delta(G)+3)^2/4$.
For cubic graphs, the above bound yields $C(G) \leq 9$.
Alikhani, Golmohammadi, and Konstantinova have found
the coalition numbers of cubic graphs of the order at most 10 \cite{1}.
They showed that $C(G) \in \{6, 7, 8 \}$ for these graphs and asked the following question.

\begin{question}
 Is it true that, for any cubic graph $G$ of order at least 6,
 its coalition number $C(G) \in \{6, 7, 8 \}$?
\end{question}

In this paper, we demonstrate that the question has the negative answer  by 
constructing an infinity family of cubic graphs with $C(G)=9$.

\section{Main result}

The main goal of this section is to present an infinite family of cubic graphs with coalition number 9.
In order to achieve this goal, we first need to find a cubic graph $G$ with $C(G)=9$.
The computational search shows that the earliest examples appear in cubic graphs of order 16.
Among 4060 cubic graphs of order 16, there are precisely 14 graphs satisfying this property.
Their diagrams are depicted in Fig.~\ref{Fig1}.
The vertices of each single-vertex set of coalition partitions are shown in black,
while the vertices of the other sets have distinct colors.
All non-one-vertex sets of coalition partitions are presented near the cor\-res\-pon\-ding graph.
It is not hard to verify that these sets form coalition partitions.

\begin{proposition} \label{5}
For any $n \ge 16$, there are cubic graphs of order $n$
with coalition number 9.
\end{proposition}
\begin{proof}
Consider cubic graphs $H$ and $G$  of order 16 and 18 shown in Fig.~\ref{Fig2}.
The graph $G$ is obtained from $H$ by subdividing edges (1,2), (2,3)
and then connecting new vertices.
Their coalition partitions are presented below the corresponding graph. 
Since every black single-vertex set 
 is not adjacent to a vertex that is either red, green, or yellow, none of the sets
 of these partitions are dominating sets.
Coalition partners in partitions $\mathcal{P}(H)$ and $\mathcal{P}(G)$
are listed in Table~\ref{Tab1}.

\begin{table}[h]
\centering
\caption{Dominating sets in graphs $H$ and $G$.} \label{Tab1}
\begin{tabular}{r|r} \hline
\rule{0cm}{2mm}  
$\mathcal{P}(H)$ \ \ \ \ \ \ \  &  $\mathcal{P}(G)$ \ \ \ \ \ \ \ \  \\ \hline
$\{2\}    \cup  \{8, 9, 10\}$   & $\{1\} \cup \{2, 7, 11, 14\}$    \\
$\{11\}   \cup  \{3, 6, 16\}$   & $\{3\} \cup \{2, 7, 11, 14\}$   \\
$\{12\}   \cup  \{3, 6, 16\}$   & $\{9\} \cup \{4, 8, 10, 18\}$   \\
$\{13\}   \cup  \{4, 5, 15\}$   & $\{12\} \cup \{5, 6, 16, 17\}$  \\
$\{14\}   \cup  \{4, 5, 15\}$   & $\{13\} \cup \{4, 8, 10, 18\}$  \\
$\{1, 7\} \cup  (8, 9, 10)$     & $\{15\} \cup \{5, 6, 16, 17\}$  \\
\hline
\end{tabular}
\end{table}

\newpage

\begin{figure}[h!]
\begin{minipage}[h]{0.49\linewidth}
\center{\includegraphics[height= 35mm,width=\linewidth]{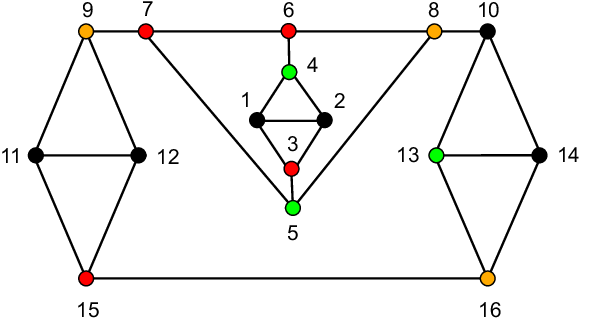}} \\
{\small $G_1$: (3  6  7 15) (4 5 13) (8  9 16)}
\end{minipage}
\hfill
\begin{minipage}[h]{0.49\linewidth}
\center{\includegraphics[width=\linewidth]{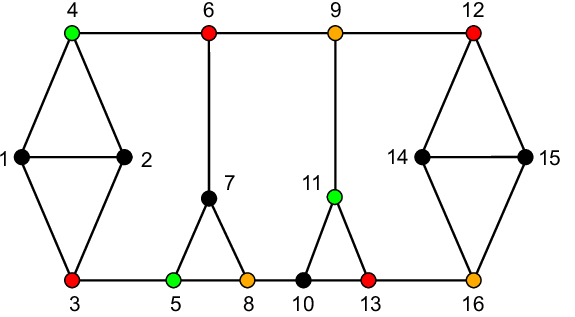}} \\
{\small $G_2$: (3  6  12 13) (4 5 11) (8 9 16)}
\end{minipage}
\vspace{4mm}
\vfill
\begin{minipage}[h]{0.49\linewidth}
\center{\includegraphics[width=\linewidth]{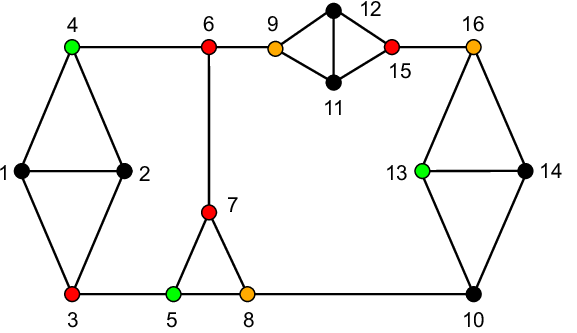}} \\
{\small $G_3$: (3  6  7 15) (4 5 13) (8 9 16)}
\end{minipage}
\hfill
\begin{minipage}[h]{0.49\linewidth}
\center{\includegraphics[width=1.05\linewidth]{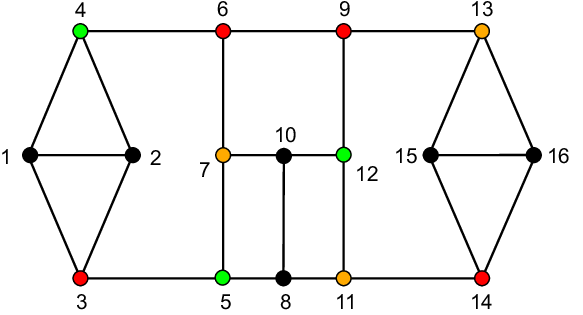}} \\
{\small $G_4$: (3  6  9 14) (4 5 12) (7 11 13)}
\end{minipage}
\vspace{4mm}
\vfill
\begin{minipage}[h]{0.49\linewidth}
\center{\includegraphics[width=\linewidth]{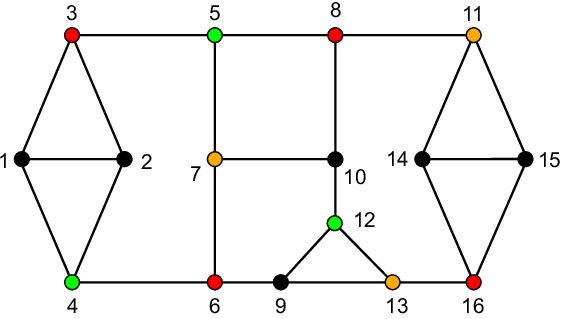}} \\
{\small $G_5$: (3  6  8 16) (4 5 12) (7 11 13)}
\end{minipage}
\hfill
\begin{minipage}[h]{0.49\linewidth}
\center{\includegraphics[width=\linewidth]{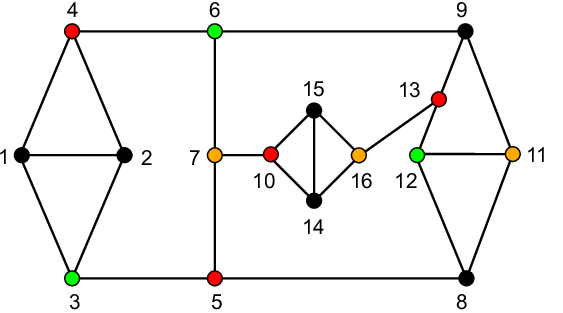}} \\
{\small $G_6$: (3  6  12) (4 5 10 13) (7 11 16)}
\end{minipage}
\vfill
\vspace*{2mm}
\begin{minipage}[h]{0.49\linewidth}
\center{\includegraphics[height= 36mm,width=\linewidth]{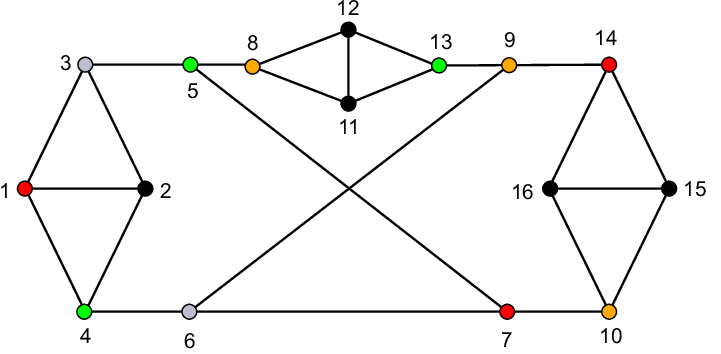}} \\
{\small $G_7$: (1 7 14) (3 6) (4 5 13) (8 9 10)}
\end{minipage}
\hfill
\begin{minipage}[h]{0.49\linewidth}
\center{\includegraphics[width=\linewidth]{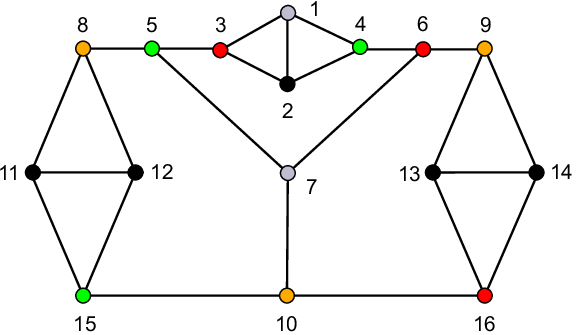}} \\
{\small $G_8$: (1 7) (3 6 16) (4 5 15) (8 9 10)}
\end{minipage}
\caption{Cubic graphs of order 16 with $C(G)=9$ and
 their coalition partitions (without single-vertex sets).}
\label{Fig1}
\end{figure}

\newpage

\setcounter{figure}{0}

\begin{figure}[h!]
\begin{minipage}[h]{0.49\linewidth}
\center{\includegraphics[width=\linewidth]{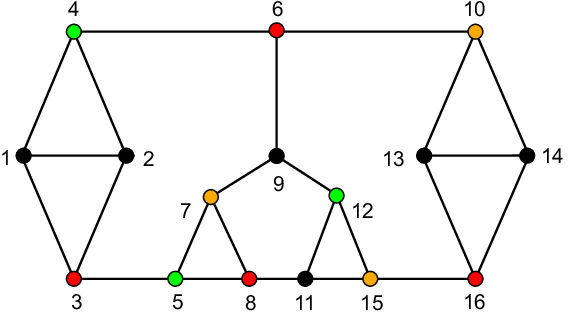}} \\
{\small $G_9$: (3  6  8 16) (4 5 12) (7 10 15)}
\end{minipage}
\hfill
\begin{minipage}[h]{0.49\linewidth}
\center{\includegraphics[width=\linewidth]{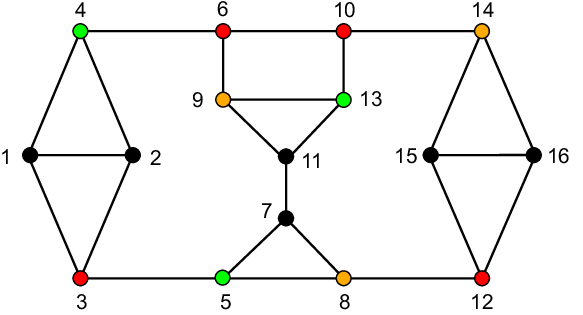}} \\
{\small $G_{10}$: (3  6  10 12) (4 5 13) (8 9 14)}
\end{minipage}
\vspace{4mm}
\vfill
\begin{minipage}[h]{0.49\linewidth}
\center{\includegraphics[width=\linewidth]{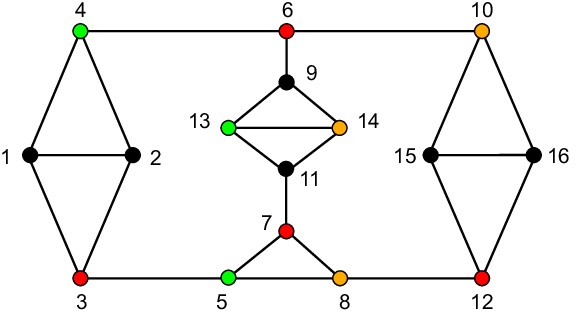}} \\
{\small $G_{11}$: (3  6  7 12) (4 5 13) (8 10 14)}
\end{minipage}
\hfill
\begin{minipage}[h]{0.49\linewidth}
\center{\includegraphics[width=\linewidth]{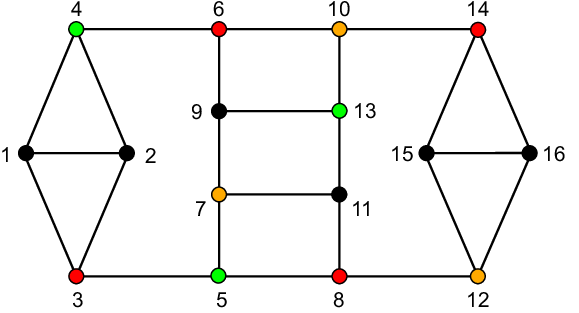}} \\
{\small $G_{12}$: (3  6  8 14) (4 5 13) (7 10 12)}
\end{minipage}
\vspace{4mm}
\vfill
\begin{minipage}[h]{0.49\linewidth}
\center{\includegraphics[height= 36mm,width=\linewidth]{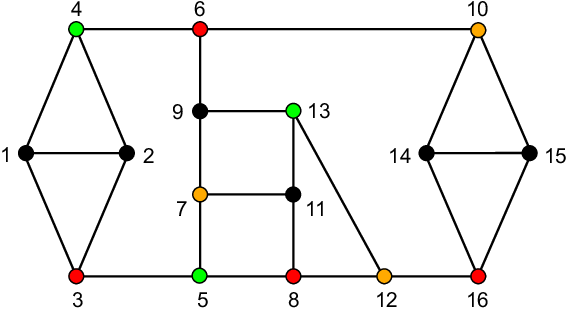}} \\
{\small $G_{13}$: (3  6  8 16) (4 5 13) (7 10 12)}
\end{minipage}
\hfill
\begin{minipage}[h]{0.49\linewidth}
\center{\includegraphics[width=\linewidth]{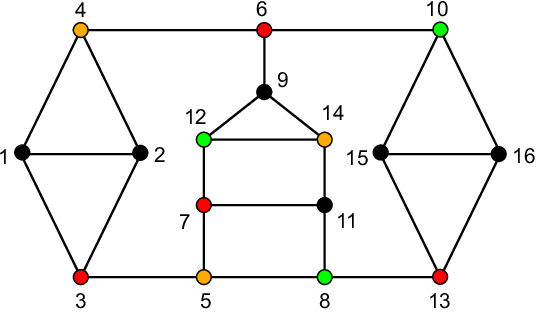}} \\
{\small $G_{14}$: (3  6  7 13) (4 5 14) (8 10 12)}
\end{minipage}
\caption{Cubic graphs of order 16 with $C(G)=9$ and
 their coalition partitions (without single-vertex sets) ({\it continue}).}
\end{figure}

Graphs $H$ and $G$ are the initial graphs of a family of new cubic graphs 
with increasing order. 
Sequentially adding the diamond graph $K_4-e$ to $H$ or $G$ instead of an edge coming from vertex 7, allows one
to construct a cubic graph of an arbitrary order with the required property.
The process of inserting colored graph $K_4-e$ into $H$ and $G$ is illustrated in 
Fig.~\ref{Fig2}
(see graphs $H_1$ and $G_1$ obtained after two steps).
The colors of each vertex and its neighbors of $H$ and $G$ remain unchanged throughout this process.
The proof is complete.
\end{proof}

It is worth mentioning that other cubic graphs $G$ of order 18 with coalition number 9 can be also obtained from
graphs of order 16 in Fig.~\ref{Fig1}. For instance, one can subdivide edges (1,2) and (2,3) of graphs $G_3, G_7, G_{11}, G_{13}$
and then connect them. 
Using these graphs, as well as graphs of order 16 in 
Fig.~\ref{Fig1},
one can construct new infinite families of cubic graphs $G$ having $C(G)=9$.



\begin{figure}[t]
\begin{minipage}[h]{\linewidth}
\center{\includegraphics[width=0.9\linewidth]{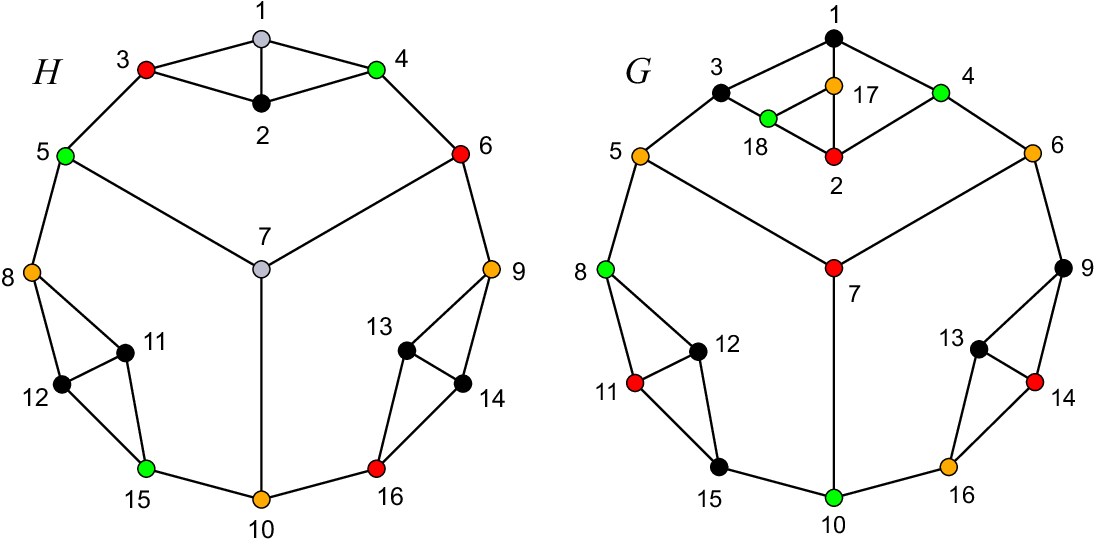}} \\
{\small $\mathcal{P}(H) = \{ (2)\ (11)\ (12)\ (13)\ (14)\ (1,7)\ (3, 6, 16)\  (4, 5, 15)\ (8, 9, 10) \}$} \\[1mm]
{\small $\mathcal{P}(G) = \{ (1)\ (3)\ (9)\ (12)\ (13)\ (15)\ (2, 7, 11, 14)\  (4, 8, 10, 18)\ (5, 6, 16, 17) \}$ }
\end{minipage}
\vspace*{2mm}
\vfill
\begin{minipage}[h]{\linewidth}
\center{\includegraphics[width=0.9\linewidth]{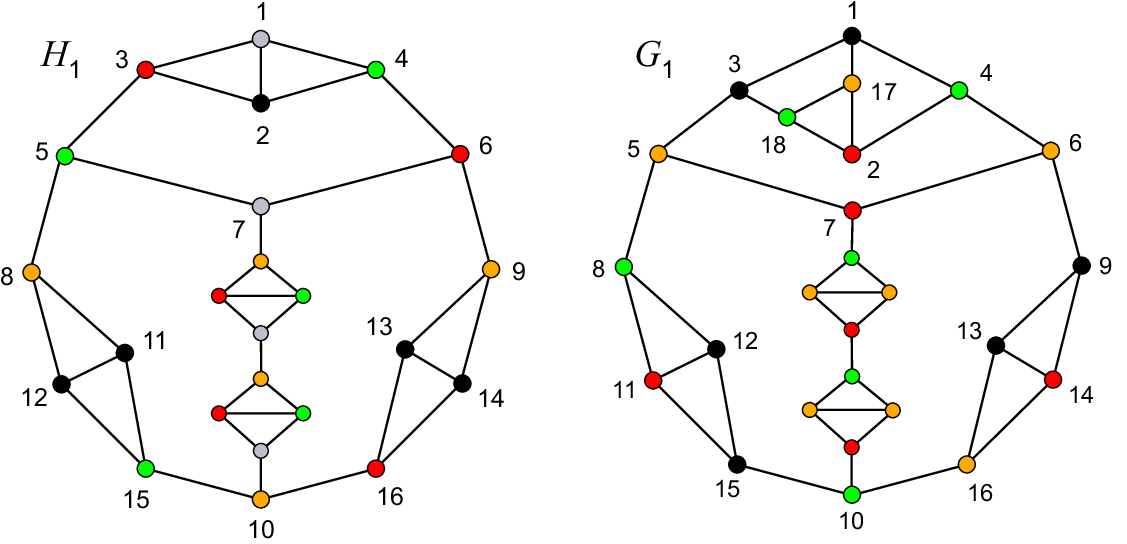}}
\end{minipage}
\caption{A family of cubic graphs with $C(G)=9$.}
\label{Fig2}
\end{figure}

Finally, we state the following question.

\begin{question}
Find $4$-regular graphs with maximum coalition number.
\end{question}

\section*{Acknowledgment}

The study of A.A.~Dobrynin was supported by the state contract
of the Sobolev Institute of Mathematics (project number FWNF-2022-0017)
and the work of Hamidreza Golmohammadi was supported by the
Mathematical Center in Akademgorodok, under agreement No. 075-15-2022-281 with
the Ministry of Science and High Education of the Russian Federation.


\bigskip
\begin{thebibliography}{99}
\bibitem{1}
 S. Alikhani, H. Golmohammadi, E. Konstantinova,
\href{https://doi.org/10.22049/cco.2023.28328.1507}
{\it Coalition of cubic graphs of order at most 10},
Commun. Comb. Optim.,
(2023).


\bibitem{2}
S. Alikhani, D. Bakhshesh, H. Golmohammadi,
{\it Total coalitions in graphs},
arXiv:2211.11590v2 [math.CO], (2022).


\bibitem{3}
S. Alikhani, D. Bakhshesh, H. Golmohammadi, S. Klavzar,
\href{https://doi.org/10.7151/dmgt.2543}
{\it On independent coalition in graphs and independent coalition graphs},
Discuss. Math. Graph Theory,
(2024).


\bibitem{4}
S. Alikhani, D. Bakhshesh, H.R. Golmohammadi, E.V. Konstantinova,
\href{https://doi.org/10.7151/dmgt.2509}
{\it Connected coalitions in graphs},
Discuss. Math. Graph Theory,
(2023).


\bibitem{5}
D. Bakhshesh, M. A. Henning, D. Pradhan,
\href{https://doi.org/10.1007/s40840-023-01492-4}
{\it On the coalition number of trees}, Bull. Malays.
Math. Sci. Soc.,
{\bf 46}:3 (2023), Paper No. 95.


\bibitem{5a}
A. Cabrera-Mart\'{\i}nez, J.A. Rodr\'{\i}guez-Vel\'{a}zquez,
\href{https://doi.org/10.1016/j.dam.2021.05.011}
{\it A note on double domination in graphs},
Discrete Appl. Math.
{\bf  300} (2021), 107--111.


\bibitem{6}
W.J. Desormeaux, T.W. Haynes, M.A. Henning,
\href{https://doi.org/10.1016/j.dam.2013.06.023}
{\it Bounds on the connected domination nuber of a graph},
Discrete Appl. Math.,
{\bf 161}:18, (2013) 2925--2931


\bibitem{7}
W. Goddard, M.A. Henning,
\href{https://doi.org/10.1016/j.disc.2012.11.031}
{\it Independent domination in graphs: A survey and recent results},
Discrete Math.,
{\bf 313}:7 (2013), 839-854.


\bibitem{8}
 H. Golmohammadi,
\href{https://doi.org/10.22049/cco.2024.29015.1813}
{\it Total coalitions of cubic graphs of at most order 10},
Commun. Comb. Optim.,
(2023).

\bibitem{8a}
F. Harary, T.W. Haynes, 
{\it Double domination in graphs},
Ars Combin.,
{\bf 55} (2000), 201--213.


\bibitem{9}
T.W. Haynes, J.T. Hedetniemi, S.T. Hedetniemi, A.A. McRae, R. Mohan,
\href{https://doi.org/10.1080/09728600.2020.1832874}
{\it Introduction to coalitions in graphs},
 AKCE Int. J. Graphs Combin.,
{\bf 17}:2 (2020), 653--659.


\bibitem{10}
 T.W. Haynes, J.T. Hedetniemi, S.T. Hedetniemi, A.A. McRae, R. Mohan,
\href{https://doi.org/10.7151/dmgt.2416}
{\it Coalition graphs of paths, cycles, and trees},
Discuss. Math. Graph Theory,
{\bf 43} (2023), 931--946.


\bibitem{11}
T.W. Haynes, J.T. Hedetniemi, S.T. Hedetniemi, A.A. McRae, R. Mohan,
{\it Upper bounds on the coalition number},
Austral. J. Combin.,
{\bf 80}:3 (2021), 442--453.


\bibitem{12}
T.W. Haynes, J.T. Hedetniemi, S.T. Hedetniemi, A.A. McRae, R. Mohan,
\href{https://doi.org/10.22049/CCO.2022.27916.1394}
{\it Coalition graphs},
Comm. Combin. Optim.,
{\bf 8}:2 (2023), 423--430.


\bibitem{13}
T. W. Haynes, S. T. Hedetniemi, M. A. Henning (eds),
\href{https://doi.org/10.1007/978-3-031-09496-5\_15}
{\it Domination in Graphs},
Core Concepts Series: Springer Monographs in Mathematics, Springer,
Cham, 2023.


\bibitem{14}
T.W. Haynes, S.T. Hedetniemi, P.J. Slater,
\href{https://doi.org/10.1201/9781482246582}
{\it Fundamentals of domination in graphs},
in: Chapman and Hall/CRC Pure and Applied Mathematics Series,
Marcel Dekker Inc., New York, 1998.


\bibitem{15}
M.A. Henning, A. Yeo,
\href{https://doi.org/10.1007/978-1-4614-6525-6}
{\it Total domination in graphs},
Series: Springer Monographs in Mathematics, Springer,
Cham, New York, 2013.
\end{thebibliography}
\end{document}